\documentclass{amsart}[12pt]
\usepackage{amsmath, amsthm, amscd, amsfonts,}
\usepackage{graphicx,float}

\usepackage{amssymb}
\usepackage{comment}

\newcommand{\PP}{{\mathbb{P}}}

\DeclareMathOperator{\otp}{otp}

\DeclareMathOperator{\acc}{acc}
\DeclareMathOperator{\Col}{Col}

\DeclareMathOperator{\image}{''}

\def\MPB{{\mathbb{P}}}
\def\MQB{{\mathbb{Q}}}

\newtheorem{theorem}{Theorem}[section]
\newtheorem{lemma}[theorem]{Lemma}

\newtheorem{definition}[theorem]{Definition}

\newtheorem{claim}[theorem]{Claim}
\newtheorem{question}[theorem]{Question}
\numberwithin{equation}{section}

\newcommand{\ZFC}{{\rm ZFC}}
\newcommand{\GCH}{{\rm GCH}}

\usepackage{pb-diagram}

\def\rmark{\mbox{$\rm\bf\rule{0.06em}{1.45ex}\kern-0.05em R$}}
\def\pmark{\mbox{$\rm\bf\rule{0.06em}{1.45ex}\kern-0.05em P$}}
\def\nmark{\mbox{$\rm\bf\rule{0.06em}{1.45ex}\kern-0.05em N$}}
\def\vdash{\mbox{$\rm\| \kern-0.13em -$}}

\def\rmark{\mbox{$\rm\bf\rule{0.06em}{1.45ex}\kern-0.05em R$}}
\def\pmark{\mbox{$\rm\bf\rule{0.06em}{1.45ex}\kern-0.05em P$}}
\def\nmark{\mbox{$\rm\bf\rule{0.06em}{1.45ex}\kern-0.05em N$}}
\def\vdash{\mbox{$\rm\| \kern-0.13em -$}}

\begin{document}
\title[The tree property ]{The tree property on a countable segment of successors of singular cardinals}

\author[ M. Golshani and Y. Hayut]{Mohammad Golshani and Yair Hayut}

\thanks{The first author's research has been supported by a grant from IPM (No. 91030417).}
\begin{abstract}
Starting from the existence of many supercompact cardinals, we construct a model of $ZFC+\GCH$ in which the tree property holds at a countable segment of successor of singular cardinals.
\end{abstract}
\maketitle

\section{Introduction}
Let $\kappa$ be a cardinal. A $\kappa$-tree $T$ is a tree with height $\kappa$ such that for every $\alpha < \kappa$, there are less than $\kappa$ elements in level $\alpha$. Moreover, we will assume always that the tree is normal, namely that for every element $t\in T$ with limit height the branch $\{s\in T\mid s < t\}$ determines $t$, and that for every $t\in T$ there are $s \in T$ above $t$ at arbitrarily high levels.
A branch through $T$ is a subset well-ordered by $<_T$ in ordertype $\kappa.$
As usual $T_\alpha$ denotes the set of all elements in $T$ with height $\alpha$.

\begin{definition}
A $\kappa$-tree $T$ is Aronszajn if it has no branches. We say that the tree property holds at $\kappa$ if there is no $\kappa$-Aronszajn tree.
\end{definition}

In this paper we will be interested in getting the tree property at a certain segment of successors of singular cardinals. In \cite{MagidorShelah96}, Magidor and Shelah showed that the tree property holds (in \ZFC) at the successor of a singular limit of strongly compact cardinals. They also showed that it is consistent relative to a large cardinal assumption, slightly above a huge cardinal, that the tree property holds at $\aleph_{\omega + 1}$. This result was later improved by reducing the large cardinal assumption to $\omega$ many supercompact cardinals by Sinapova \cite{Sinapova2012a} using the diagonal Prikry forcing, and Neeman \cite{Neeman2014} using  the product of ordinary collapse posets. In this paper we extend this result and show that it is consistent, relative to the existence of large cardinals, that the tree property holds at a countable segment of successor of singular cardinals. More precisely, we prove the following:
\begin{theorem}\label{thm: main theorem}
Assume there are $\kappa^{+}$ many supercompact cardinals, where $\kappa$ is a supercompact cardinal. Then for each countable ordinal $\vartheta,$  there exists a model of $ZFC+\GCH$ in which
the tree property holds at all cardinals of the form $\aleph_{\omega \cdot \alpha + 1}$, where  $0 < \alpha < \vartheta$.
\end{theorem}
In section 2, we will present and quote some technical lemmas that will be used during the proof of the main theorem. Section 3 is devoted to the proof of Theorem~\ref{thm: main theorem}.

\section{Some preservation lemmas}
When constructing models of the tree property, especially at successor of singular cardinals, it is sometimes easier to prove that a certain $\nu^+$-Aronszajn tree has a cofinal branch in some generic extension of our model. In order to show that this branch exists in the ground model we use preservation lemmas.

The next preservation lemma is due to Magidor-Shelah (see \cite[Theorem 2.1]{MagidorShelah96}).
\begin{lemma}\label{lem: preservation lemma}
Let $\mu$ be a singular cardinal of cofinality $\omega$, and suppose that $\PP$ and $\MQB$ are two forcing notions such that $|\PP|=\rho < \mu$ and
$\MQB$ is $\rho^+$-closed. Assume $\dot{T}$ is a $\mathbb{P}$-name for a $\mu^+$-tree. Then forcing with $\MQB$ over $V^\mathbb{P}$ does not add a branch to $\dot{T}$.
\end{lemma}

We will also use a variant of the following theorem, due to Neeman, \cite{Neeman2014}:
\begin{theorem}
Let $\langle \mu_n \mid n < \omega\rangle$ be an increasing sequence of indestructible supercompact cardinals. There is $\rho < \mu_0$ such that the forcing $\Col(\omega, \rho) \times \Col(\rho^{+}, {<}\mu_0) \times \prod \Col(\mu_n, {<}\mu_{n+1})$ forces the tree property at $\big(\sup_{n < \omega} \mu_n\big)^+$.
\end{theorem}
Note that the cardinals $\langle \mu_n \mid n < \omega\rangle$, are preserved by the forcing notion $\Col(\rho^{+}, {<}\mu_0) \times \prod \Col(\mu_n, {<}\mu_{n+1})$ and so $\big(\sup_{n < \omega} \mu_n\big)^+$ becomes $\rho^{+\omega+1}$ by this forcing.

As indicated in \cite{Neeman2014}, one can replace the forcing notions  of the form $\Col(\mu, {<}\nu)$ with any $\mu$-closed $\nu$-Knaster forcing notion of cardinality $\nu$. This argument is due to Neeman, and for the completeness of this paper we will show how it works in our case.

We will need to replace the Levy collapse with a forcing that has better projection properties. Let us say that a cardinal $\gamma$ is \emph{strong regular} if $\gamma^{<\gamma}=\gamma.$
\begin{definition}[Shioya \cite{shioya}]
The Easton collapse $E(\mu, \nu)$ is the product with Easton support $\prod_{\mu \leq \gamma < \nu,\ \gamma\text{~strong regular}} \Col(\mu, \gamma)$.
\end{definition}
\begin{lemma}[Shioya \cite{shioya}]
Let $\mu$ be regular and $\nu$ be a Mahlo cardinal. Then $E(\mu, \nu)$ is $\mu$-closed, $\nu$-Knaster and $|E(\mu,\nu)| = \nu$.
\end{lemma}

\begin{lemma} \label{lemma: projection Shioya}
Let $\langle \mu_i \mid i < \zeta\rangle$ and $\langle \nu_i \mid i < \zeta\rangle$ be increasing sequences of regular cardinals, and assume that $\zeta < \mu_0$ and $\mu_i \leq \nu_i < \mu_{i+1}$. Let $\lambda \geq \sup \mu_i$.

There is a projection from $E(\mu_0, \lambda)$ onto the full support product $\prod_{i < \zeta} E(\mu_i, \nu_i)$.
\end{lemma}
\begin{proof}
For every $\alpha < \beta < \gamma$ regular, $\gamma$ strong regular, there is a continuous projection from $\Col(\alpha, \gamma)$ to $\Col(\beta, \gamma)$ since $\Col(\beta, \gamma)$ is an $\alpha$-closed forcing notion of cardinality $\gamma$. Let $\rho^{\beta, \gamma}_{\alpha}$ be such a projection.

For $p \in E(\mu_0, \lambda),$ set
\[
\pi(p) = \langle  \langle \rho^{\mu_i, \gamma}_{\mu_0}(p (\gamma) ) \mid \gamma \in   [\mu_i, \nu_{i} )\rangle             \mid i < \zeta        \rangle
\]
Then $\pi(p) \in \prod_{i < \zeta} E(\mu_i, \nu_{i})$ and $\pi:E(\mu_0, \lambda) \to  \prod_{i < \zeta} E(\mu_i, \nu_{i})$ is easily seen to be a projection.
\end{proof}

We now prove Neeman's theorem when replacing the Levy collapse by the Easton collapse.

The following technical lemma shows that the conditions of \cite[Lemma 3.10]{Neeman2014} are satisfied.
\begin{lemma}\label{lem: technial lemma}
Let $\kappa$ be indestructible supercompact and let $\mu < \kappa$ be regular. Let $\mathbb{P}_0, \mathbb{P}_1$ be forcing notions, $|\mathbb{P}_0| \leq \mu$, $\mathbb{P}_0$ is $\mu$.c.c., and $\mathbb{P}_1$ is $\kappa$-directed closed. Then in the generic extension by $\mathbb{P}_0 \times E(\mu, \kappa) \times \mathbb{P}_1$, $\kappa$ is generically supercompact by a forcing $\mathbb{R} \in V$ such that $\mathbb{R}^{\mu},$  the $\mu$-fold product of $\mathbb{R}$ with supports of size $< \mu$, is $\mu$-distributive in the generic extension.
\end{lemma}
\begin{proof}
Let $G_0 \subseteq \mathbb{P}_0$, $G_1 \subseteq \mathbb{P}_1$ and $H\subseteq E(\mu, \kappa)$ be mutually generic filters. In $V[G_1]$, $\kappa$ is supercompact by indestructibility and in $V[G_1][G_0]$ it is still supercompact by the Levy-Solovay argument.

Let $\lambda \geq \kappa$ be a regular cardinal. Let $j\colon V[G_0][G_1] \to M$ be a $\lambda$-supercompact elementary embedding with critical point $\kappa$. In particular, $\sup j\image \lambda < j(\lambda)$.
$j(E(\mu, \kappa)) \cong E(\mu, \kappa) \times \mathbb{R}$ where $\mathbb{R}$ is the Easton support product of $\Col(\mu, \gamma)$ over every $\gamma \in [\kappa, j(\kappa))$ strong regular in $M$. $\mathbb{R}$ is $\mu$-closed in $V$ and $\mathbb{R}^{\mu} \cong \mathbb{R}$.
Since $j(p) = p$ for every $p\in E(\mu, \kappa)$, in order to extend the elementary embedding $j$ we can pick any $M[H]$-generic filter for $\mathbb{R}$.

$\mathbb{R}$ is $\mu$-closed in $V[G_1][H]$ and therefore it is $\mu$-distributive in $V[G_0][G_1][H]$. Since $\mathbb{R}^\mu\cong \mathbb{R}$ the same holds for $\mathbb{R}^\mu$.
\end{proof}

Let $\langle \kappa_n \mid n < \omega\rangle$ be an $\omega$-sequence of indestructible supercompact cardinals. Let $\mathbb{C} = \prod_{n < \omega} E(\kappa_n^{++}, \kappa_{n+1})$. For every $n < \omega$, \[\mathbb{C} = \mathbb{C} \restriction n \times E(\kappa_{n}^{++}, \kappa_{n + 1}) \times \mathbb{C} \restriction [n + 1, \omega).\] By Lemma \ref{lem: technial lemma}, in the generic extension by $\mathbb{C}$ for every $n > 0$ and for every $\lambda \geq \kappa_n$, there is a forcing notion $\mathbb{R}$ that adds an elementary embedding with critical point $\kappa_n$, discontinuity at $\lambda$ and $\mathbb{R}^{\kappa_{n-1}}$ is $\kappa_{n-1}$-distributive. By the indestructibility of $\kappa_0$, it is still supercompact in the generic extension by $\mathbb{C}$.

We are now ready to apply the general result, \cite[Lemma 3.10]{Neeman2014} and conclude:
\begin{theorem}\label{thm: neeman}
Let $\langle \kappa_n \mid n < \omega\rangle$ be an $\omega$-sequence of indestructible supercompact cardinals. There is $\rho < \kappa_0$ such that the forcing:
\[\Col(\omega, \rho) \times \Col(\rho^{+}, <\kappa_0) \times \prod_{n < \omega} E(\kappa_n^{++}, \kappa_{n+1})\]
forces the tree property at the successor of $\sup \kappa_n$.
\end{theorem}
We may note that   the cardinals $\rho^+$ and $\kappa_n, \kappa_n^+, \kappa_n^{++}, n<\omega,$ are preserved
by $\Col(\omega, \rho) \times \Col(\rho^{+}, <\kappa_0) \times \prod_{n < \omega} E(\kappa_n^{++}, \kappa_{n+1}),$
in particular, it forces $\sup \kappa_n = \aleph_\omega.$

\section{Getting tree property at many successor of singular cardinals}
In this section we prove our main theorem 1.2.

Let $S$ be a set of indestructible supercompact cardinals and assume that $\otp S = (\min S)^{+}$. Let $\delta = \sup S$ and let $\kappa_0 = \min S$.
\begin{comment}
\begin{lemma}
There is an unbounded $\sigma$-closed set $D \subseteq \delta$ and a sequence $\langle s_\alpha \mid \alpha\in D\rangle$ such that:
\begin{itemize}
\item Each $s_\alpha$ is an $\omega$-sequence cofinal in $\alpha.$
\item For all $\alpha, \beta \in D \cap S^\delta_\omega, s_\alpha \cap s_\beta$ is an initial segment of both $s_\alpha$ and $s_\beta.$
\item For all $\alpha \in D \cap S^\delta_\omega, \min(s_\alpha)=\min(S)=\kappa_0$.
\end{itemize}
\end{lemma}
\begin{proof}
Let $b\colon S^{<\omega} \to S$ be a bijection. Let $D$ be the set of points of countable cofinality such that for every $\zeta \in D$, $b\restriction \zeta^{<\omega} \colon (S \cap \zeta)^{<\omega} \to S \cap \zeta$ is a bijection. We assume that $b(\emptyset) = \min S$ and that if $t_0 \trianglelefteq t_1$ then $b(t_0) \leq b(t_1)$.

Let us pick a cofinal sequence $r_\alpha\subseteq S$ below each $\alpha \in D$. Let $s_\alpha(n) = b(r_\alpha \restriction n)$.

Let $\alpha, \beta\in D$ and assume that $\gamma \in s_\alpha \cap s_\beta$. Then for some $m,n < \omega, $
\[
b(r_\alpha \restriction n) = s_\alpha(n)=\gamma = s_\beta(m)= b(r_\beta \restriction m).
\]
As $b$ is an injection, $r_\alpha \restriction n=r_\beta \restriction m$ and in particular $m=n$. Then
 \[
 s_\alpha \restriction n = \{ b(r_\alpha \restriction k)  \mid k < n             \}=  \{ b(r_\beta \restriction k)  \mid k < n             \}=s_\beta \restriction n,
 \]
as wanted.
\end{proof}
\end{comment}

Let $D = \acc S \cap S^\delta_\omega,$ where $\acc S$ is the set of accumulation points of $S$ and $S^\delta_\omega=\{\alpha < \delta: cf(\alpha)=\omega       \}$, and let us pick for every $\alpha \in D$ an $\omega$-sequence $s_\alpha \subseteq S$ such that $\sup s_\alpha = \alpha$ and $\kappa_0 = \min s_\alpha$.

For every $\alpha \in D$, by Theorem \ref{thm: neeman}, there is $\rho_\alpha < \kappa_0$ such that the forcing
\[
\MPB_\alpha= \Col(\omega, \rho_\alpha) \times \Col (\rho_\alpha^{+}, {<}\kappa_0) \times \prod E(s_\alpha(n)^{++}, s_\alpha(n+1))
\]
forces the tree property at $\alpha^{+}$. Therefore, there is $\rho < \kappa_0$ and $S' \subseteq D$ stationary, such that for every $\alpha \in S'$,
$\rho_\alpha = \rho$.

There is $n < \omega$ such that for every club $E \subseteq \delta$, $\sup \{s_\alpha(n) \mid \alpha \in S' \cap E\}$ is $\delta$ (where $s_\alpha(n)$ is the $n$-th element in the increasing enumeration of $s_\alpha$). Let $n_0$ be this natural number. Using Fodor's lemma again, we may narrow down $S'$ and fix the first $n_0$ elements in $s_\alpha$ for every $\alpha \in S'$.

\begin{lemma}
Suppose $\vartheta < \omega_1.$ Then there exists a closed $t \subseteq S'$ of order type $\vartheta$ such that
\begin{center}
$\alpha \in acc(t) \Longrightarrow \alpha \in S'$ and $s_\alpha \subseteq t.$
\end{center}
\end{lemma}
\begin{proof}
We prove the claim by induction on $\vartheta$. We strengthen the induction hypothesis to assert also the for every $\gamma < \delta$ there is $t$ such that also the $n_0 + 1$-th element of $t$ is at least $\gamma$ and the first $n_0$ elements are the fixed first $n_0$ elements of the sequences $s_\alpha$ for $\alpha \in S'$.

Let us observe that if $t$ satisfies the conditions of the lemma, then so does any $t'$ such that $t \subseteq t' \subseteq S'$, and $|t' \setminus t| < \aleph_0$ with the same first $n$ elements. Therefore, the interesting cases are when $\vartheta$ is a limit ordinal.

If $\vartheta = \varphi + \omega$, then by the induction hypothesis there is $t'$ of order type $\varphi$, satisfying the conditions of the lemma. Let $\alpha > \sup t'$ in $T$. Take $t = t' \cup \{\sup t'\} \cup s_\alpha$.

Assume that $\vartheta$ is a limit of limit ordinals and let $\gamma_0 < \delta$. Let us pick $\alpha \in S'$ such that for every $\zeta < \vartheta$ and every $\beta < \alpha$ there is $t\subseteq \alpha$ of order type $\zeta + 1$ satisfying the conditions of the lemma, with $t(n_0+1) \geq \beta$. Moreover, $s_\alpha(n_0 + 1) \geq \gamma_0$. Since $T$ is stationary and by the choice of $n_0$ - this is possible. Let $\langle \beta_n \mid n < \omega\rangle$ be the increasing enumeration of $s_\alpha$. Let $\langle \zeta_n \mid n < \omega\rangle$ be a sequence of countable ordinals such that $\sup \zeta_n = \vartheta$.

By induction on $n < \omega$, let $t_n$ be a witness for the conditions of the lemma such that $\otp t_n \geq \zeta_n$, $t_{n+1}(n_0 + 1) > \max t_{n}, s_\alpha(n)$, $t_0(n_0) \geq \gamma_0$ and $t_n \subseteq \alpha$. The sequence $t = s_\alpha \cup \bigcup_{n < \omega} t_n$ has ordertype $\geq \vartheta$ and satisfies the conditions of the lemma, since $\acc t = \{\alpha\} \cup \bigcup_{n < \omega} \acc t_n$.
\end{proof}

We are now ready to prove the main theorem 1.2.
\begin{proof}[Proof of Theorem 1.2]
Let $D$, $\langle s_\alpha \mid \alpha \in D\rangle$ and $t$ be as above.

Let $\MPB_t$ be the forcing notion
\[
\MPB_t= \Col(\omega, \rho) \times \Col(\rho^{+}, {<}\kappa_0) \times \prod_{i < \vartheta} E(t(i)^{++}, t(i+1))
\]
and let $G_t$ be $\MPB_t$-generic over $V$. We show that in $V[G_t],$ the tree property holds at all cardinals of the form $\aleph_{\alpha+1},$ where $\alpha < \vartheta$
is a limit ordinal. The next claim is clear.
\begin{claim}
For each limit ordinal $\alpha < \vartheta, \aleph_{\alpha}^{V[G_t]} = t(\alpha)$ and $\aleph_{\alpha+1}^{V[G_t]} = t(\alpha)^{+}.$
\end{claim}

Assume $\alpha < \vartheta$ is a limit ordinal, so, by choice of  $t$, we have $s_{t(\alpha)} \subseteq t$. Also let $\langle \xi_n \mid n < \omega \rangle$ be
an increasing sequence of ordinals less than $\vartheta$ such that $t(\xi_n)=s_{t(\alpha)}(n)$ (thus $\xi_0=0$). Also set $\xi=\sup_{n<\omega}\xi_n \leq \vartheta.$ Note that
\[
t(\xi) = \sup_{n<\omega} t(\xi_n) = \sup_{n<\omega}s_{t(\alpha)}(n) = t(\alpha),
\]
so in fact $\xi=\alpha.$
Then we can write $\MPB_t$ as $\MPB_t = \MPB_{t, 1} \times \MPB_{t, 2}$
where
\[
\MPB_{t, 1} =  \Col(\omega, \rho) \times \Col(\rho^{+}, < \kappa_0) \times \prod_{i < \alpha}  E(t(i)^{++},  t(i+1))
\]
and
\[
\MPB_{t, 2} = \prod_{\alpha \leq i < \vartheta}  E(t(i)^{++},  t(i+1)).
\]
Also let $G_t=G_{t,1} \times G_{t,2}$ correspond to $\MPB_t = \MPB_{t, 1} \times \MPB_{t, 2}$.
We have
\[
\prod_{i < \alpha} E(t(i)^{++},  t(i+1)) = \prod_{n<\omega} \prod_{\xi_n \leq i < \xi_{n+1}} E(t(i)^{++},  t(i+1)).
\]
But, by Lemma \ref{lemma: projection Shioya}, for any $n<\omega,$ there is a natural projection
\[
\pi_n: E(s_{t(\alpha)}(n)^{++},  s_{t(\alpha)}(n+1)) \to \prod_{\xi_n \leq i < \xi_{n+1}} E(t(i)^{++},  t(i+1)),
\]
and so we have a projection
\[
\prod_{n<\omega}\pi_n:  \prod_{n<\omega} E(s_{t(\alpha)}(n)^{++},  s_{t(\alpha)}(n+1)) \to
  \prod_{n<\omega} \prod_{\xi_n \leq i < \xi_{n+1}} E(t(i)^{++},  t(i+1)).
\]
This induces a projection (in $V$)
\[
\pi: \MPB_{t(\alpha)} \to \MPB_{t, 1}.
\]

Let us split $\MPB_{t(\alpha)}$ into a product $\Col(\omega, \rho) \times \MQB_{t(\alpha)}$ where $\MQB_{t(\alpha)}$ is $\rho^+$-closed. Similarly, split $\MPB_{t,1} = \Col(\omega, \rho) \times \MQB_{t, 1}$. Let $H_{t,1}$ be the generic filter $G_{t,1}$ restricted to $\MQB_{t,1}$.

The quotient forcing $\mathbb{R}= \MPB_{t(\alpha)}/ G_{t, 1}$ is isomorphic to $\MQB_{t(\alpha)} / H_{t,1}$. By the continuity of the projection and the closure of the forcing notions, in $V[H_{t,1}]$, $\mathbb{R}$ is $\rho^{+}$-closed.

Let $H_{t,2}$ be a generic filter for $\mathbb{P}_{t,2}$. Let $G_\alpha$ be $\MPB_{t(\alpha)}$-generic over $V$ with $\pi[G_\alpha] = G_{t,1}.$ Let $R\subseteq \mathbb{R}$ and let $C\subseteq \Col(\omega, \rho)$ be generic filters such that $V[G_t] = V[C][H_{t, 1}][H_{t, 2}]$ and $V[G_\alpha] = V[H_{t,1}][C][R]$.

Now suppose that $T$ is a $t(\alpha)^+$-tree in $V[G_t].$  The forcing $\MPB_{t,2}$ is $t(\alpha)^{++}$-closed in $V$ and thus $t(\alpha)^{++}$-distributive in the generic extension by $\mathbb{P}_{t,1}$. So $T \in V[G_{t,1}].$ The tree property holds at $t(\alpha)^+$ in $V[G_\alpha]$, so  $T$ has a cofinal branch in $V[G_\alpha]$, and by applying Lemma \ref{lem: preservation lemma} in $V[H_{t, 1}]$, $\mathbb{R}$ cannot add a branch to  $T$. So we conclude that $T$ already has a branch in $V[G_{t, 1}] \subseteq V[G_t]$. So $T$ is not a $t(\alpha)^+$-Aronszajn tree in $V[G_t]$, as required.
\end{proof}

We close the paper with the following question.
\begin{question}
Is it consistent with $ZFC$ that the tree property holds at successor of every singular cardinal $?$
\end{question}

\subsection*{Acknowledgements}
The first author's research was in part supported by a grant from IPM (No. 95030417). The authors would like to thank the referee of the paper for his/her  helpful suggestions.

School of Mathematics, Institute for Research in Fundamental Sciences (IPM), P.O. Box:
19395-5746, Tehran-Iran.

E-mail address: golshani.m@gmail.com

The Hebrew University of Jerusalem, Einstein Institute of Mathematics,
Edmond J. Safra Campus, Givat Ram, Jerusalem 91904, Israel.

E-mail address:  yair.hayut@mail.huji.ac.il
\end{document}